\documentclass{amsart}
\usepackage[utf8]{inputenc}
\usepackage{amsmath,amssymb}
\usepackage{xcolor}

\newcommand{\Z}{\mathbb{Z}}
\newcommand{\F}{K\langle X|G\rangle}
\newcommand{\fx}{f(x_1,\ldots,x_m)}

\newtheorem{thm}{Theorem}
\newtheorem{lemma}[thm]{Lemma}
\newtheorem{proposition}[thm]{Proposition}
\newtheorem{conjecture}{Conjecture}
\newtheorem{remark}[thm]{Remark}
\newtheorem{example}{Example}
\newtheorem{definition}{Definition}

\title[Images of graded polynomials on matrix algebras]{Images of graded polynomials on matrix algebras}

	\author[L. Centrone]{Lucio Centrone}
	\address{Dipartimento di Matematica, Universit\`a degli Studi di Bari, via Orabona 4, 70125, Bari, Italy}\email{lucio.centrone@uniba.it} \address{IMECC, Universidade Estadual de
		Campinas, Rua S\'ergio Buarque de Holanda, 651, Cidade Universit\'aria ``Zeferino Vaz'', Distr. Bar\~ao Geraldo, Campinas, S\~ao Paulo, Brazil, CEP
		13083-859}\email{centrone@unicamp.br}
	
	\author[T. C. de Mello]{Thiago Castilho de Mello}
	\address{Instituto de Ci\^encia e Tecnologia, Universidade Federal de S\~ao Paulo, Av. Cesare M. Giulio Lattes, 1201, 12247014,	S\~ao Jos\'e dos Campos, SP, Brazil}\email{tcmello@unifesp.br}

\begin{document}

\begin{abstract}
    The aim of this paper is to start the study of images of graded polynomials on full matrix algebras. We work with the matrix algebra $M_n(K)$ over a field $K$ endowed with its canonical $\mathbb{Z}_n$-grading (Vasilovsky's grading). We explicitly determine the possibilities for the linear span of the image of a  multilinear graded polynomial over the field $\mathbb Q$ of rational numbers and state an analogue of the  L'vov-Kaplansky conjecture about images of multilinear graded  polynomials on $n\times n$ matrices, where $n$ is a prime number. We confirm such conjecture for polynomials of degree 2 over $M_n(K)$ when $K$ is a quadratically closed field of characteristic zero or greater than $n$ and for polynomials of arbitrary degree over matrices of order 2. We also determine all the possible images of semi-homogeneous graded  polynomials evaluated on $M_2(K)$. 
\end{abstract}
\maketitle

\section{Introduction}

Let $K$ be a field and $K\langle X \rangle$ be the free associative algebra freely generated by the set $X=\{x_1, x_2, \dots, \}$ over $K$, i.e., $K\langle X  \rangle$ is the algebra of noncommutative polynomials in the variables of $X$ and with coefficients from $K$. If $A$ is a $K$-algebra, a polynomial $f=f(x_1, \dots, x_m)$ defines a map (also denoted by $f$):
	\[\begin{array}{cccc}
		f: & A^m & \longrightarrow & A \\
		& (a_1,\dots,a_m) & \longmapsto & f(a_1,\dots,a_m) \\
	\end{array}
	\]
The image of such map is called the \textit{image} of the polynomial $f$ on $A$.%

Recently, images of polynomials have been studied by several authors, mostly motivated by the famous open problem known as \emph{L'vov-Kaplansky Conjecture}. Such problem asks whether the image of a multilinear polynomial on the matrix algebra $M_n(K)$ is a vector subspace of $M_n(K)$. In this case, it must be one of the following subspaces: the full matrix algebra $M_n(K)$, the set of traceless matrices $sl_n(K)$, the set of scalar matrices that we identify with the ground field $K$, or the set $\{0\}$.

A solution to the L'vov-Kaplansky Conjecture is known only for some values of $m$ or $n$. For instance, the case $m=2$ is a consequence of a well-known result of Shoda \cite{Shoda} and Albert and Muckenhoupt \cite{AM} which states that any trace zero matrix is given by a commutator. Some partial results for $m=3$ were obtained in \cite{Dykema_Klep}. For $n=2$ a solution was given in \cite{K-BMR} for $K$ a quadratically closed field and in \cite{MalevM} for $K=\mathbb{R}$. The case $n=3$ has a partial solution too and we address the reader to the paper \cite{K-BMR3}.

It is easy to achieve that the analogue of the L'vov-Kaplansky Conjecture fails to be true if $A$ is not simple or $A$ is not finite dimensional. For, if $A$ is the infinite dimensional Grassmann algebra generated by $\{e_1, e_2, \dots\}$, the elements $e_1e_2$ and $e_3e_4$ lie in the image of the commutator $f(x_1,x_2)=[x_1,x_2]:=x_1x_2-x_2x_1$, but their sum does not.
Furthermore, in \cite{SantuloYukihide} the authors provide an example of a non-simple finite dimensional algebra whose images on a certain class of polynomials are not subspaces.

Some generalizations of the L'vov-Kaplansky Conjecture have been studied considering algebras other than $M_n(K)$. The possible images of a multilinear polynomial are known for the algebra of upper triangular matrices $UT_n(K)$ and for its subalgebra of strictly upper triangular matrices \cite{GargatedeMello, Wang_nxn, Fagundes} and also for the algebra of quaternions \cite{MalevQ} and for some classes of simple Jordan algebras \cite{MalevJ}.

The theory of polynomial identities in algebras (PI-theory) and the study of images of polynomials on algebras have a strong connection. Polynomials whose image is $\{0\}$ are the so called \textit{polynomial identities} of $A$ and those whose image is $K$, are the so called \textit{central polynomials} of $A$. Also, the solution of the case $n=2$ of the L'vov-Kaplansky Conjecture relies on the fact that the ideal of polynomial identities of $M_n(K)$ is a prime ideal in $K\langle X \rangle$.

An important tool in the study of polynomial identities are $G$-graded identities on $G$-graded algebras, where $G$ is a group. In the celebrated work of Kemer \cite{Kemer}, a crucial role was played by the $\Z_2$-graded identities in the solution of the Specht Problem. After the publication of Kemer's theory, a large number of papers on graded identities and graded central polynomials have been published, specially after the seminal papers \cite{div1} and \cite{Vasilovsky}.

In the light of the above facts, we consider a natural step to study images of graded polynomials on algebras.

Up to our knowledge, there is only one paper published toward images of graded polynomials on full matrix algebras: the one written by Kulyamin (see \cite{Kulyamin}). In that paper the author considers the algebra $A=M_n(K[G])$, where $K[G]$ is a finite group algebra of an abelian group $G$ over $K$ endowed with the natural $G$-grading on $A$ is induced by the grading on $K[G]$. The author proves that a homogeneous subset $S\subseteq A$ is the image of a graded polynomial with zero constant term if and only if $0\in S$ and $S$ is invariant under conjugation by degree zero elements of $A$. It is worth mentioning that during the preparation of this paper, the authors became aware of the preprint \cite{PlamenPedro} where the authors consider images of graded polynomials on upper triangular matrices.

In this paper, we study images of graded polynomials on $A=M_n(K)$ endowed with the canonical $\mathbb{Z}_n$-grading $A=\oplus_{g\in \mathbb{Z}_n}A_g$.
This case is completely different from Kulyamin's, since here the algebra $A$ is simple, while $M_n(K[G])$ is not simple, once $K[G]$ is not.
We believe our results can be generalized for more general types of gradings on $M_n(K)$.

The paper is organized as follows: first we present the basic definitions and results to study the problem. Later, we prove that the linear span of a  multilinear graded polynomial on a $\mathbb{Q}$-algebra $A$ is one of the following: $A_g$, for some $g\in \mathbb{Z}_n$, $\mathbb{Q}$ (viewed as the set of scalar matrices), $(sl_n)_0$, the set of trace zero diagonal matrices, or $\{0\}$. In light of this result we state a conjecture regarding the image of a  multilinear graded polynomial on $M_n(K)$ and we prove this conjecture for $n=2$ in the case $K$ is a quadratically closed field. The paper ends with a description of images of semi-homogeneous graded polynomials on $M_2(K)$.

\section{Preliminaries}
In this paper, all fields we refer to are assumed to be of characteristic zero and all algebras we consider are associative and unitary.
If $n$ is a positive integer and $1\leq i,j\leq n$, we denote by $E_{i,j}$ the matrix units, i.e., $E_{i,j}$ is the matrix whose entry $(i,j)$ is 1 and all other entries are $0$. If $k,l$ are not in the interval $[1,\dots,n]$, $E_{k,l}$ is defined by considering the representative of $k$ and $l$ modulo $n$ in $[1,\dots, n]$.

Let $G$ be any group and let $K$ be a field. When we consider an arbitrary group, we will use the multiplicative notation and we will denote the group unit by $1$. When considering an abelian group, we will use the additive notation and denote its unit by $0$. 

If $A$ is a $K$-algebra, we say $A$ is a \textit{$G$-graded algebra} if there are subspaces $A_g$, for each $g\in G$, such that \[A=\bigoplus_{g\in G} A_g \ \textrm{and for each } g,h\in G, \ A_gA_h\subseteq A_{gh}.\]
If $0\neq a\in A_g$, we say $a$ is \emph{homogeneous of $G$-degree g} and we write $\deg(a)=g$. We shall denote by $h(A)$ the set of homogeneous elements of the graded algebra $A$.

\begin{example}
\begin{enumerate}
    \item Any algebra may be endowed with a trivial $G$-grading, where $G$ is any group, if we set $A_{1}=A$ and for each $g\neq 1$ $A_g=0$. 
    \item If $A=K[G]$ is the group algebra generated by $G$ over the field $K$, $A$ is naturally $G$-graded if we set $A_g=K\cdot g$ for each $g\in G$. 
    \item If $A=M_n(K)$
    and $G$ is a group, let $\overline{g}=\{g_1,\ldots,g_n\}$ be an $n$-tuple of elements of $G$, then $A$ is $G$-graded 
    if we set $A_g$ to be the subspace generated by matrices $E_{ij}$ such that $g_i^{-1}g_j=g$. 
    This grading is called the \emph{elementary grading determined by $\overline g$}.
    
    \item If in the above example we set $G=\Z_n$, and we choose the $n$-tuple $\overline{g}$ to be $(\overline{0},\overline{1},\ldots,\overline{n-1})$, we will refer to this grading as the \textit{Vasilovsky grading} of $A$. Notice that the component $A_0$ is the set of diagonal matrices.
    We recall that such grading was introduced by Di Vincenzo in \cite{div1} for $2\times 2$ matrices, and its graded polynomial identities were described in \cite{div1} for $n=2$ and in the general case in \cite{Vasilovsky} (characteristic zero) and \cite{Azevedo} (positive characteristic). Moreover,  the graded central polynomials in this case were described in \cite{Brandao}.

\end{enumerate}
\end{example}

In order to work in the setting of graded algebras, we need to introduce the graded analogue of polynomials, the so called \textit{graded polynomials}.

Let $\{X_{g}\mid g \in G\}$ be a family of disjoint countable sets. Set $X=\bigcup_{g\in G}X_{g}$ and denote by $\F$ the free associative algebra freely generated by the set $X$ over $K$. To define a grading on $\F$ we first put $\deg(x)=g$, if $x\in X_{g}$ and we extend this map to monomials by setting \[\deg(x_{i_1}x_{i_2}\cdots x_{i_k})=\deg(x_{i_1})\cdot\deg(x_{i_2})\cdots\deg(x_{i_k}). \] We will say $x_{i_1}\cdots x_{i_k}$ has \textit{$G$-homogeneous degree $g$} (or $G$-degree $g$, or homogeneous degree $g$). 
For every $g \in G$ we denote by $\F_g$ the subspace of $\F$ spanned by all monomials having homogeneous $G$-degree $g$. Notice that $\F_g\F_{h}\subseteq \F_{gh}$ for all $g,h \in G$. Thus \[\F=\bigoplus_{g\in G}\F_g\] is a $G$-graded algebra. We refer to the elements of $\F$ as \textit{$G$-graded polynomials} or just \textit{graded polynomials}. An ideal $I$ of $\F$ is said to be a $T_{G}$-ideal if it is invariant under all $K$-endomorphisms $\varphi:\F\rightarrow\F$ such that $\varphi\left(\F_g\right)\subseteq\F_g$ for all $g\in G$. If $A$ is a $G$-graded algebra, a $G$-graded polynomial $\fx$ is said to be a \emph{graded polynomial identity} of $A$ if $f(a_1,a_2,\ldots,a_m)=0$ for all $a_1,a_2,\dots,a_m\in h(A)$ such that $a_k\in A_{\deg(x_k)}$, $k=1,\dots,m$. 
If $A$ satisfies a non-trivial graded polynomial identity, $A$ is said to be a \textit{$G$-graded PI-algebra}. We denote by $T_G(A)$ the ideal of all graded polynomial identities of $A$. It is a $T_G$-ideal of $\F$.

Let $A$ be a $G$-graded algebra. If $\fx\in \F$, let $g_i=\deg(x_i)\in G$. Then $f$ defines a map (also denoted by $f$):
	\[\begin{array}{cccc}
		f: & A_{g_1}\times \cdots \times A_{g_m} & \longrightarrow & A \\
		& (a_1,\dots,a_m) & \longmapsto & f(a_1,\dots,a_m) \\
	\end{array}
	\]
\begin{definition}
The image of such map is called the image of the graded polynomial $f$ on the graded algebra $A$.
\end{definition}

Below, we can find some examples.

\begin{example}
\begin{enumerate}
    \item If $A$ is a $G$-graded algebra, and $\fx\in \F$ is a graded polynomial, then $Im(f)=\{0\}$ if and only if $\fx$ is a graded polynomial identity of $A$.
    \item If $A=M_n(K)$ with the Vasilovsky grading, then the image of the polynomial $f(x_1,\dots,x_n)=\sum_{\sigma\in S_n}x_{\sigma(1)}\cdots x_{\sigma(n)}$ is $K$ (the set of scalar matrices) if $\deg(x_1)=\cdots = \deg(x_n)=\overline{1}$ (see \cite[Proposition 1]{Brandao}).
    \item If $A=UT_n(K)$ is the set of $n\times n$ upper triangular matrices, endowed with the Vasilovsky (induced) grading, then the image of $f(x_1,\dots,x_n)=x_1\cdots x_{n-1}$, with $\deg(x_1)=\cdots \deg(x_{n-1})=\overline{1}$, is the 1-dimensional subspace of $A$ spanned by ${E_{1n}}$.
\end{enumerate}
\end{example}

 We say a $G$-graded polynomial $p\in \F$ is \textit{multilinear} of degree $n$ if it is multilinear as a polynomial, that is, if it can be written as \[\sum_{\sigma\in S_n} \alpha_\sigma x_{\sigma(1)} \cdots x_{\sigma(n)}, \] for some $\alpha_\sigma\in K$. 

We also say a $G$-graded polynomial $p(x_1, \dots, x_m)$  is \textit{multihomogeneous} of degree $(n_1, \dots, n_m)$ if the variable $x_i$ appears exactly $m_i$ times in any of its monomials.

We call the reader's attention to the fact that there are two different gradings been considered in $\F$.
A $G$-grading, as defined above and the usual $\mathbb{Z}$-multigrading. In particular, a multilinear graded polynomial is a polynomial in $K\langle x_1, \dots, x_m | G\rangle $ which is multihomogeneous of degree  $(1, \dots, 1)$.
    
\section{The linear span of the image of a homogeneous graded polynomial}

In this section we study the linear span of the image of a  $\mathbb{Z}_n$-homogeneous graded polynomial on $M_n(K)$ endowed  with the canonical $\mathbb{Z}_n$-grading. 

The next is a straightforward adaptation of \cite[Lemma 5]{K-BMR}.

\begin{lemma}\label{units}
    Let $f(x_1,\dots, x_m)$ be a $\mathbb{Z}_n$-graded polynomial of homogeneous degree $g\neq 0$. Assume that $a_1, \dots, a_m$ are matrix units in $M_n(K)$. Then $f(a_1,\dots, a_m)$ is a scalar multiple of $E_{i,j}$, for some $i\neq j$ such that $\deg (E_{i,j})=g$.
\end{lemma}

Since $G$ is abelian, conjugation by a homogeneous element does not change the degree of a homogeneous element
We have the following result. 

\begin{lemma}\label{cone}
    Let $f(x_1,\dots, x_m)$ be a graded polynomial of homogeneous $\mathbb{Z}_n$-degree $g$. Then $Im (f)$ is a subset of $M_n(K)$ which is invariant under conjugation by homogeneous matrices. In particular, it is invariant under conjugation by $N=\sum_{i=1}^n E_{i,i+1}$ and $D=\sum_{i=1}^n d_{i}E_{i,i}$, for any $d_1, \dots d_n\in K\setminus \{0\}$. 
\end{lemma}

\begin{lemma}
    Let $f(x_1,\dots, x_m)$ be a graded polynomial of homogeneous $\mathbb{Z}_n$-degree $g\neq 0$. If $f$ is not a graded polynomial identity of $A=M_n(K)$, then the linear span of $Im (f)$ equals the homogeneous component $A_g$.
\end{lemma}

\begin{proof}
    By Lemma \ref{units}, there exists an evaluation of $f$ equal to $c\cdot E_{i,j}$, for some $i\neq j$, $c\neq 0$, with $\deg (E_{i,j})=g$. Hence $E_{i,j} \in Im(f)$. By Lemma \ref{cone}, $Im (f)$ is invariant under conjugation by $N$. Then $N^{-1}E_{i,j}N=E_{i+1,j+1} \in Im(f)$. By applying the same argument $n$ times, we obtain $E_{i+k,j+k} \in Im(f)$, for any $k$ and the proof is complete.   
\end{proof}

The following is a consequence of Lemma \ref{cone}.

\begin{lemma}\label{conjugating}
    Let Let $A=M_n(K)$ be endowed with the canonical $\mathbb{Z}_n$-grading and let $M = \sum_{i=1}^n \gamma_i E_{i, i+g}\in A_g$, where $g$ is invertible in $\mathbb{Z}_n$. 
    \begin{enumerate}
        \item If $\gamma_i\neq 0$ for every $i\in\{1,\dots,n\}$, then there exists $D\in A_0$ such that all entries of  $DMD^{-1}$ but one are equal to 1.
        \item If $\gamma_i=0$, for some $i\in \{1, \dots, n\}$, there exists $D\in A_0$ such that all entries of $DMD^{-1}$ are 1 or 0.
    \end{enumerate}
\end{lemma}

\begin{proof} Since $g$ has a multiplicative inverse in $\mathbb{Z}_n$, we have
    \[\{kg\,|\, k\in\{1,\dots,n\}\} = \mathbb{Z}_n,\]
    then, if $M\in A_g$, we may write $M=\sum_{k=1}^n \gamma_{k}E_{kg,(k+1)g}$  and if $D\in A_0$, we may write $D=\sum_{k=1}^n d_{k}E_{kg,kg}$. 
    Direct computations show that if $D$ is invertible,
    \[DMD^{-1} = \sum_{k=1}^n {\frac{d_k}{d_{k+1}} \gamma_k} E_{kg,(k+1)g}\] 
    Now one can directly verify that if all $\gamma_k$ are nonzero, the system of equations 
    \[\frac{d_k}{d_{k+1}} \gamma_k =1,\quad \text{ for }  k\in\{1, \dots, n-1\}\]
    has a solution by defining {$d_1= 1$, and $d_k = \prod_{i=1}^{k-1} \gamma_i$}, for $k\in \{2, \dots, n-1\}$.
    
    In a similar way, one can find  a solution to the system of equations
    \[\frac{d_k}{d_{k+1}} \gamma_k = 1, \quad\text{ for } k\in\{1, \dots, n-1\}, \text{ with } \gamma_k\neq 0.\]
\end{proof}

The next two results show an analogue of a well known theorem of Shoda, Albert and Muckenhoupt (see \cite{Shoda, AM}) in the graded case, i.e., it describes the image of a (graded) commutator polynomial. 

\begin{lemma}\label{commutator}
    Let $C=\sum_{i=1}^n c_iE_{i, i+g}\in  A_g$, for some $g\neq 0$ in $\mathbb{Z}_n$. Then there exists $B\in A_g$ and $D\in A_0$ such that $C=[B,D]$.
\end{lemma}

\begin{proof}
    Let $C=\sum_{i=1}^n c_iE_{i, i+g}\in A_g$ and consider $D=\textrm{diag}(d_1, \dots, d_n)$, where $d_1, \dots, d_n$ are pairwise distinct elements in $K$. Write $B=\sum_{i=1}^nb_iE_{i, i+g}$; direct computations show
    \[[B, D]=\sum_{i=1}^nb_i(d_{i+g}-d_i)E_{i, i+g}.\]
    By defining $b_i=(d_{i+g}-d_i)^{-1}c_i$ for each $i$, we obtain $C=[B,D]$ and we are done.
\end{proof}

    We now turn our attention to polynomials of homogeneous degree zero.

\begin{proposition}\label{tr0}
    Let $D\in A_0$ such that $tr(D)=0$. If $g$ is invertible in $\mathbb{Z}_n$, then there exist $B\in A_g$, $C\in A_{-g}$ such that $D=[B,C]$.
\end{proposition}

\begin{proof}
        Write $D=\textrm{diag}(d_1,\dots,d_n)$, $B=\sum_{i=1}^n{b_i}E_{i,i+g}\in A_g$ and $C=\sum_{i=1}^n E_{i+g,i}\in A_{-g}$. Then
        \[ [B,C] = \sum_{i=1}^n (b_i-b_{i-g}) E_{ii}.\]
        The lemma will be proved once we show the system of equations
        \[b_{i}-b_{i-g}=d_i, \quad i=1, \dots, n\] in the variables $b_1, \dots, b_n$ has a solution.
        
        To do that, notice that since $g$ is invertible in $\mathbb{Z}_n$, we have \[\{1-kg\,|\, k\in\{0,\dots, n-1\}\}=\mathbb{Z}_n.\] Given $b_1\in K$, by defining
        \begin{align*}
            b_{1-g}  &=  b_1-d_1\\
            b_{1-2g} &= b_{1-g}-d_{1-g} = b_1-d_1-d_{1-g}\\
           & \vdots \\
            b_{1-(n-1)g} &=b_{1-(n-2)g}-d_{1-(n-2)g} = b_1-(d_1+\cdots + d_{1-(n-2)g})
        \end{align*}
            we obtain a solution to  the above system and we are done.
        \end{proof}
        
\begin{remark}
    In the previous result, the condition that $g$ has a multiplicative inverse in $\mathbb{Z}_n$ cannot be removed. For instance, let $n=4$ and $g=\overline 2\in \mathbb{Z}_n$. Simple computations show that the image of the graded polynomial $[x_1,x_2]=x_1x_2-x_2x_1$, where $\deg(x_1)=\deg(x_2)=\overline{2}$ lies in the set $\{D=\text{diag}(d_1, d_2, -d_1, -d_2)\,|\, d_1, d_2\in F\}$, which does not contain all traceless diagonal matrices.
\end{remark}

\begin{proposition}
    Let us denote by $K$ the field $\mathbb{Q}$ of rational numbers. Let $n$ be a prime number and let $f$ be a multilinear $\mathbb{Z}_n$-graded polynomial of degree $0$. Then the linear span, $L$, of $Im (f)$ on  $A=M_n(K)$ (endowed with the Vasilovsky $\mathbb{Z}_n$-grading) is one of the following:
\[\{0\}, \quad K, \quad (sl_n(K))_0, \quad A_0\]
\end{proposition}

\begin{proof}

    If $f$ is not an identity and the image of $f$ contains only scalar matrices, then of course $L=K$.
    
    Assume now that $f$ is not an identity nor a central polynomial of $M_n(K)$. 
    
    Since $f$ is not an identity nor a central polynomial, there exists an evaluation of $f$ which is a nonscalar matrix.
    
    Let us first assume there is an element $D=\sum_{i=1}^n\alpha_i E_{i,i}$ in the image, satisfying $\sum \alpha_i\neq 0$. Write $D_0=D, D_1=N^{-1}DN, \cdots, D_{n-1} = N^{-(n-1)}DN^{n-1}$. Once we show that $D_0, \dots, D_{n-1}$ are linearly independent, they generate an $n$-dimensional subspace in $L$. Since $L$ is a vector subspace of the $n$-dimensional space $A_0$, we will conclude that $L=A_0$. The elements $D_0, \dots, D_{n-1}$ can be considered as $n$-tuples of elements of $K$. if $D=D_0=(\alpha_1, \dots, \alpha_n)$, then $D_1=(\alpha_n, \alpha_1, \dots, \alpha_{n-1}), \dots, D_{n-1}=(\alpha_2, \dots, \alpha_{n}, \alpha_1)$. These are linearly independent over $K$, if and only if the determinant of the matrix below is nonzero.
      \[C=\begin{pmatrix} \alpha_1 & \alpha_2 & \cdots & \alpha_{n-1} & \alpha_n \\
                        \alpha_n & \alpha_1 & \cdots & \alpha_{n-2} & \alpha_{n-1} \\
                         \vdots & \vdots&  \ddots & \vdots & \vdots\\
                         \alpha_ 3 & \alpha_4  & \cdots & \alpha_1 & \alpha_2\\
                        
                        \alpha_2 & \alpha_3 & \cdots  & \alpha_n & \alpha_1 \\
                        \end{pmatrix}.\]
                        
    Matrices of the above type are called \emph{circulant matrices}. It is well known that $C$ is nonsingular if and only if the polynomial $P(x)=\alpha_1+\alpha_2x+\cdots +\alpha_nx^{n-1}\in K[x]$ is coprime to the polynomial $x^n-1$ (see for instance Corollary 10 of \cite{Kra}). Since $\alpha_1+\cdots + \alpha_n\neq 0$, $x=1$ is not a root of $P(x)$. Since $n$ is prime, $\frac{x^n-1}{x-1}$ is irreducible over $K$, and this implies $C$ is nonsingular and the linear span of $Im (f)$ is $A_0$.
    
    Finally, we need to prove that if $Im (f)$ contains only trace zero matrices, then $Im (f)$ is $(sl_n)_0$. We argue as above. The only difference is that we need to prove that the elements $D_0, \dots, D_{n-1}$ generate an $n-1$-dimensional subspace of $A_0$. This is equivalent to show that the rank of $C$ is $n-1$, but it is well known that the rank of the circulant matrix $C$ above is $n-d$, where $d$ is the greatest common factor of $P(X)$ and $x^n-1$ (again we address the reader to the paper \cite{Kra}). Now we have $\alpha_1+\cdots +\alpha_n=0$, $x=1$ is a root of $P(X)$, then because $\frac{x^n-1}{x-1}$ is irreducible over $\mathbb{Q}$, we obtain $d=1$ and the rank of $C$ is $n-1$. As a consequence, the linear span of $Im(f)$ is $(sl_n)_0$.
    \end{proof}
    
    \begin{remark}
        The above proof holds only for $K=\mathbb{Q}$. It would be interesting to prove it for an arbitrary field $K$.
    \end{remark}
    
    The analogue of the L'vov-Kaplansky conjecture in the graded case (for matrices of prime order) can now be stated as
    \begin{conjecture}
        Let $K$ be a field and $n$ be a prime number. If $f\in K\langle X| \mathbb{Z}_n\rangle$ is a multilinear $\mathbb{Z}_n$-graded  polynomial then $Im(f)$ is one of the following:
        
\[\{0\},\quad  K, \quad  sl_n(K)_0, \quad \text{ or } \quad M_n(K)_g, \text{ for some } g\in G.\]

    \end{conjecture}

\section{The image of multilinear polynomials of degree 2 on matrix algebras}
We are going to deal with multilinear polynomials of degree 2 over quadratically closed fields of sufficiently large characteristic.

\begin{lemma}\label{finally}
    Let $n$ be an odd prime number and $K$ be a quadratically closed field of characteristic zero or greater than $n$. Let $f(x_1,x_2)=x_1x_2-\alpha x_2x_1\in K\langle x_1,x_2\rangle$ be a graded polynomial with $\deg (x_1) = g\in \mathbb{Z}_n$ and $\deg (x_2) = h\in \mathbb{Z}_n$  and $\alpha$ be an $n$-th root of 1. If $f$ is not a graded polynomial identity for $M_n(K)$, then the image of $f$ on $M_n(K)$ is $A_{g+h}$. 
\end{lemma}

\begin{proof}
     If $g=h=0$, then modulo the graded identities of $M_n(K)$, $f(x_1,x_2) = (1-\alpha)x_1x_2$. Then $f$ is a graded identity if $\alpha = 1$ and the image of $f$ is $A_{0}$ otherwise. Hence, we may assume $g$ or $h$ is nonzero. Let us assume without loss of generality $h\neq 0$. Write $B=\sum_{i=1}^n b_i E_{i,i+g}$ and $C=\sum_{i=1}^{n} c_i E_{i,i+h}$. Then \[f(B,C) = BC-\alpha CB = \sum_{i=1}^n(b_ic_{i+g} - \alpha c_ib_{i+h}) E_{i,i+g+h}.\]
    Let $M=\sum_{i=1}^n \gamma_i E_{i,i+g+h}\in A_{g+h}$. We will show that there exist $B$ and $C$ as above such that $f(B,C)=M$.
    
    Observe that the the above claim is equivalent to find a solution to the system of equations
    \begin{align}\label{system}
        b_i c_{i+g} - \alpha c_i b_{i+h} = \gamma_i, \quad i\in \{1, \dots, n\}.
    \end{align}

First we notice that if $g=0$, the above system of equation become
 \begin{align}\label{system1}
        c_i(b_i  - \alpha b_{i+h}) = \gamma_i, \quad i\in \{1, \dots, n\}.
    \end{align}
then a solution to the above system of equations can be found in a similar way to the proof of Lemma \ref{commutator}.

So we may assume from now on that $g$ and $h$ are different from zero in $\mathbb{Z}_n$. Notice that if all $\gamma_i$ are zero, $B=C=0$, provide a solution to the above system. Also, if we admit only one entry to be nonzero, a solution to the above equation may be found. Indeed, by Lemma \ref{cone}, we may assume such nonzero entry is $\gamma_0$. In this case, $b_0=1$, $b_i=0$, for $i>1$, $c_g=\gamma_0$ and $c_i=0$, for $i\neq g$ provide a solution to the system of equations (\ref{system1}). So, from now on we may assume at least two of the $\gamma_i$ in $M$ are nonzero.

Since $n$ is prime, and $h\neq 0$, we have in $\mathbb{Z}_n$, $\{1, \dots, n\} = \{0, h, \dots, (n-1)h\}$. Hence the above system of equations can be written as 
    \[b_{kh} c_{kh+g} - \alpha c_{kh} b_{(k+1)h} = \gamma_{kh}, \quad k\in \{0, \dots, n-1\}.\]
    
If we assume $c_i\neq 0$ for all $i$, we have 
\begin{align*}\label{recursion}
b_{(k+1)h} = \dfrac{b_{kh}c_{kh+g}-\gamma_{kh}}{\alpha c_{kh}}, \quad k\geq 0.
\end{align*}

An easy computation shows that for every $k\geq 1$, we have 

\begin{equation}\label{b}
    b_{kh} = \dfrac{b_0 \displaystyle\prod_{i=0}^{k-1} c_{g+ih} - \sum_{r=0}^{k-1} \alpha^r\gamma_{rh} \prod_{i=0}^{r-1}c_{ih} \prod_{i=r+1}^{k-1}c_{g+ih}}{\displaystyle \alpha^k\prod_{i=0}^{k-1}c_{ih}}
\end{equation} 

The above expressions will provide a solution to the system of equations (\ref{system1}), but we need to take into account that condition $b_0 = b_{nh}$ holds. 

Substituting $k$ by $n$ in Equation (\ref{b}), since $nh= 0 \mod n$ and $\alpha^n=1$, we obtain

\begin{align*}
     b_{0} = \dfrac{b_0 \displaystyle\prod_{i=0}^{n-1} c_{g+ih} - \sum_{r=0}^{n-1} \alpha^r\gamma_{rh} \prod_{i=0}^{r-1}c_{ih} \prod_{i=r+1}^{n-1}c_{g+ih}}{\displaystyle\alpha^n\prod_{i=0}^{n-1}c_{ih}};
\end{align*}
since $\displaystyle\prod_{i=0}^{n-1} c_{g+ih} = \displaystyle\prod_{i=0}^{n-1}c_{ih}$, we get:
\begin{align}\label{zero}
    \sum_{r=0}^{n-1} \alpha^r\gamma_{rh} \prod_{i=0}^{r-1}c_{ih} \prod_{i=r+1}^{n-1}c_{g+ih} = 0.
\end{align}

Now, showing the existence of a solution $b_1, \dots, b_{n}$ is equivalent to show that there exist $c_1, \dots, c_n\in K\setminus \{0\}$, such that Equation (\ref{zero}) holds.

Let $g=th$, with $t\in \{1, \dots, n-1\}$ (recall that $n>2$). By letting $c_i=1$, for $i \neq th, (t+1)h$,  Equation (\ref{zero}) becomes

\[\left(\sum_{r=0}^{t-1}\alpha^r\gamma_{rh}\right) + \alpha^t\gamma_t c_{th} + \left(\sum_{r=t+1}^{n-1}\alpha^r\gamma_{rh}\right)c_{th}c_{(t+1)h}=0\]

Now notice that if we set $c_{(t+1)h} = c_{th}  = x$, the above become
\begin{align}\label{quadratic} 
    \left(\sum_{r=0}^{t-1}\alpha^r\gamma_{rh}\right) + \alpha^t\gamma_t x + \left(\sum_{r=t+1}^{n-1}\alpha^r\gamma_{rh}\right)x^2=0
\end{align}
and it is enough to show it has a nonzero solution.

Recall that the image of a polynomial is invariant under conjugation by homogeneous invertible matrices (Lemma \ref{cone}). Hence, we may assume $\gamma_t \neq 0$, since conjugating by some power of the matrix $N=\sum_{i=1}^nE_{i,i+1}$, one of the nonzero {entries} of $M$ lies in line $th$ (recall that the case where all $\gamma_i$ are zero has already been considered).

Also, by Lemma \ref{conjugating} we may assume  $\gamma_{rh}\in \{0,1\}$ if $r\neq t$. As a consequence, since we are considering the case in which  at least two entries of $M$ are nonzero, we obtain that at least two of the coefficients of equation (\ref{quadratic}) are nonzero. Indeed, if $\alpha=1$, the coefficients are a sum of less than $n$ times the unity 1, and if $\alpha$ is a primitive $n$-th root of 1, its minimal polynomial over the prime subfield of $K$ is $1+ x+\cdots + x^{n-1}$. This proves Equation (\ref{quadratic}) has a nonzero solution.
\end{proof}

We are now in position to give a complete description of images of multilinear $\mathbb{Z}_n$-graded  polynomials of degree 2 on the matrix algebra $M_n({K})$, provided $n$ is a prime number and $K$ is a quadratically closed field of characteristic greater than $n$.

\begin{thm}
    Let $f(x_1,x_2)$ be a multilinear $\mathbb{Z}_n$-graded  polynomial and $n$ be an odd prime number. Then the image of $f$ evaluated on the $\mathbb{Z}_n$-graded algebra $A=M_n(K)$ is one of the following:
    \[\{0\}, (sl_n)_ 0, \text{ or } A_g, \text{ for some } g\in {G} \]
\end{thm}

\begin{proof}
    Assuming $f$ is nonzero, we may write $f(x_1,x_2)=x_1x_2-\alpha x_2x_1$, for some $\alpha\in K$. Let $\deg(x_1)=g$ and $\deg(x_2)=h$. If $g=h=0$ then modulo the graded identities of $M_n(K)$, we have $x_1x_2=x_2x_1$. Hence the image of $f$ on $M_n(K)$ equals the image of the polynomial $(1-\alpha)x_1x_2$. Of course, such image is $\{0\}$ if $\alpha=1$, and is $A_0$ if $\alpha \neq 1$. Hence we may assume that $g$ or $h$ is not equal to $ 0$. Without loss of generality, we may assume $h\neq 0$. 
    
    The case $\alpha=0$ is trivial: the image is $A_{g+h}$, so we consider $\alpha \neq 0$. Now we have two more cases to consider, namely $\alpha$ is an $n$-th root of 1 or not.

    If $\alpha$ is a root of 1, Lemma \ref{finally} settles the case.
    
    So it remains to consider the case $\alpha$ is not a root of 1.

    If $M=\sum_{i=1}^n \gamma_i E_{i,i+g+h}\in A_{g+h}$, we need to find $B\in A_g$ and $C\in A_h$ such that $f(B,C)=M$.
    
    For, let us write $B=\sum_{i=1}^n b_i E_{i,i+g}$ and $C=\sum_{i=1}^n E_{i,i+h}$. Direct computations show that
    \[f(B,C)=\sum_{i=1}^n (b_i - \alpha b_{i+h})E_{i,i+g+h}.\]
    In order to obtain $B$ and $C$ satisfying $f(B,C)=E$, we need to find $b_1, \dots, b_n\in K$ such that 
    \begin{equation}\label{eq}
        b_i-\alpha b_{i+h}=\gamma_i, \qquad i\in\{1, \dots, n\}
    \end{equation}
   Since $h$ is nonzero, it is invertible in $\mathbb{Z}_n$, because $n$ is prime. As a consequence,
    \[ \{1, \dots, n\} = \{0, h, \dots, (n-1)h\}.\]
    and \[\{b_1, \dots, b_n\} = \{b_0, b_{h}, \dots, b_{(n-1)h}\}.\]
    Now the system of equations (\ref{eq}) can be written as
    
    \begin{equation}
        b_{kh}-\alpha b_{(k+1)h}=\gamma_{kh}, \qquad k\in\{0, \dots, n-1\}
    \end{equation}

    and a solution is given by
    
    \[b_1=\frac{\gamma_0+\alpha\gamma_{h}+\cdots+\alpha^{n-1}\gamma_{(n-1)h}}{1-\alpha^n}\]
    \[b_{kh}=\dfrac{b_0-(\gamma_0+\alpha\gamma_{h}+\cdots + \alpha^{k-1}\gamma_{(k-1)h})}{\alpha^k}, \quad k\in \{1,\dots, n-1\}\]
    
    Observe that since $\alpha$ is not a root of 1, $1-\alpha^n\neq 0$ and the above expressions are well defined.
\end{proof}

\begin{remark}
    Although the above theorem was stated only for odd prime numbers, there is an analogue of it for $n=2$. It will be stated in the next section. The only difference is that we may have a central polynomial in this case, i.e., the image of $f$ can also be $K$. For instance, the polynomial $p(x_1,x_2) = x_1x_2 +x_2x_1$, with $\deg(x_1) = \deg(x_2) = 1$ is a nontrivial graded central polynomial of $M_2(K)$.
\end{remark}

\section{The image of multilinear graded polynomials on $2\times 2$ matrix algebras}
In this section we compute explicitly the image of any  multilinear graded polynomial evaluated on $2\times2$ matrix algebras endowed with the Vasilovsky's grading.

\begin{definition}
Let $A$ be a $G$-graded algebra. We say $A$ is a \textit{$G$-graded domain} if $ab=0$ implies $a=0$ or $b=0$ for $a,b\in h(A)$. Moreover, we say $A$ is a \textit{$G$-graded division algebra} if any homogeneous element is invertible in $A$.
\end{definition}

Let  $A=M_n(K)$ be $\mathbb{Z}_n$-graded with the Vasilovsky grading. We consider \[Y=\{y_{ij}^{(r)}| i,j\in\{1,\ldots,n\}, r\in\mathbb{N}\}\] and for every $r\in\mathbb{N}$, let us consider the generic $n \times n$ matrices with entries from the algebra of the commutative polynomials $K[Y]$:
\begin{eqnarray*}
\xi_g^{r} &=& \sum_{j-i=g}y_{ij}^{(r)}E_{i,j}.
\end{eqnarray*}

We shall denote by $Gen^{\mathbb{Z}_n}(A)$ the algebra generated by the $\xi_g^r$'s and call it \textit{$\mathbb{Z}_n$-graded generic matrix algebra} associated to the Vasilovsky ${\mathbb{Z}_n}$-grading of $M_n(K)$.  The algebra $Gen^{\mathbb{Z}_n}(A)$ is called the graded generic algebra of $A$. It is well known \[Gen^{\mathbb{Z}_n}(A)\cong \dfrac{\F}{(\F\cap T_G(M_n(K))}.\]
The following is well-known 

\begin{proposition}\label{domain}
    The algebra $Gen^{\mathbb{Z}_n}(A)$ defined above is a graded domain.
\end{proposition}

    \begin{lemma}\label{nonsingular}
	    Let $f(x_1,\dots,x_m)$ be a ${\mathbb{Z}_n}$-graded homogeneous polynomial of degree $g\neq 0$. If $f$ is not a  graded polynomial identity for $M_n(K)$, with the Vasilivsky grading, then there exists a non-singular matrix in the image of $f$.
    \end{lemma}

    \begin{proof}
    	Assume that for any admissible evaluation we have a singular matrix. This implies that $f(x_1, \dots, x_m)^n$ is a graded identity. As a consequence, the image of $f$ via the canonical homomorphism onto the algebra of generic matrices is a nilpotent element. Since the algebra of $\mathbb{Z}_n$-graded generic matrices (with the Vasilovsky's grading) is a graded domain due to Proposition \ref{domain}, we obtain that $f$ is a graded identity. That is a contradiction.
    \end{proof}

    We now recall Lemma 1.34 of \cite{survey}, which will help us in proving the main result of this section.
    
    \begin{lemma}\label{dim2}
        Let $V_i$ (for $1\leq i\leq m$) and $V$ be linear spaces over an arbitrary field $K$. Let $f: \prod\limits_{i=1}^m V_i\rightarrow V$ be a multilinear map. Assume there exist two points in $Im(f)$ which are not proportional. Then $Im(f)$ contains a $2$-dimensional plane. In particular, if $V$ is $2$-dimensional, then $Im(f)=V$.
    \end{lemma}

We are now in position to prove the main result of this section.

    \begin{thm}
        Let $f$ be a  multilinear graded polynomial. Then the image of $f$ on $M_2(K)$ is one of the following:
        \[\{0\}, \quad K, \quad (sl_2)_0, \quad A_0, \quad A_1.\]
    \end{thm}
    
    \begin{proof}
        Let us assume $f$ is not a graded identity. If $f$ has homogeneous degree $g\neq 0$, Lemmas \ref{units} and \ref{cone} guarantee that $E_{1,2}\in Im(f)$. By Lemma \ref{nonsingular}, there exists a nonsingular matrix $B$ in $Im(f)$. The matrices $B$ and $E_{12}$ are clearly not proportional, then Lemma \ref{dim2} implies $Im(f)=A_1$.
        
        Assume now $f$ has homogeneous degree $0$. If all evaluations of $f$ are scalar matrices, we obtain  $Im(f)$ is $K$, the set of scalar matrices (since it is one-dimensional). If all evaluations of $f$ have trace zero, then $Im(f)$ is $(sl_2)_0$ because it is one dimensional.
        
        Finally, if $Im(f)$ contains a nonzero nonscalar element $B=aE_{11}+bE_{22}$ with nonzero trace, then conjugation by $N=E_{12}+E_{21}$ yields $bE_{11}+aE_{22}$, which still lies in $Im(f)$ by Lemma \ref{cone} and since $B$ is not scalar and has nonzero trace, $B$ and $N^{-1}BN$ are linearly independent. Again, Lemma \ref{dim2} completes the proof in this case and we are done.
    \end{proof}

\section{The image of semi-homogeneous graded polynomials on $2\times 2$ matrix algebras}
In this last section we address a similar question as above for semi-homogeneous polynomials. From now on any field is considered to be quadratically closed. 

We start the section with some concepts and results that will be used in the proof of the main result. We introduce now the so called \textit{$G$-graded prime} algebras. On this purpose, we address the reader to the paper \cite{balaba} by Balaba. 

We recall that an ideal of a graded algebra $A$ is a \emph{graded ideal} if it is generated by homogeneous elements. A graded ideal $P$ of $A$ is said \textit{$G$-graded prime} or \textit{$G$-prime} if it is prime as a graded ideal. Moreover, an element of $A$ is said \textit{regular} if it is not a zero divisor.

\begin{definition} A $G$-graded ideal $P$ of a $G$-graded algebra $A$ is called \textit{strongly graded prime} or $G$-\textit{strongly prime} if whenever $aAb\subseteq P,$ where $a,b\in h(A),$ either $a\in P$ or $b\in P$. Moreover, a graded algebra $A$ is called \textit{graded prime} if $(0)$ is a strongly graded prime ideal of $A$.\end{definition}

For example, every prime algebra $A$ graded by a group $G$ is $G$-prime. For suppose conversely that there exists $a,b\in h(A)$ which are not 0 and such that $aAb=0$, then $A$ is not prime.

\begin{definition}
A $G$-prime algebra $A$ is called \textit{PI $G$-prime} if it satisfies an ordinary polynomial identity.
\end{definition}

Recall that by \cite[Proposition 1]{balaba}, the localization $A_S$ of $A$ over $S$, where $S$ is a set of homogeneous regular elements of the center $Z(A)$ of $A$ is a PI $G$-graded algebra of central quotients of $A$. An algebra $Q(A)\supseteq A$ is called the \textit{left (right) graded algebra of quotients} of $A$ if:
\begin{enumerate}
\item each homogeneous regular element from $A$ is invertible in $Q(A);$
\item each homogeneous element $x\in Q(A)$ has the form $a^{-1}b$ ($ba^{-1}$), where $a,b\in h(A)$ and $a$ is regular.
\end{enumerate}

We have the following results (see \cite{balaba}). We recall if $A$ is $G$-graded, then we denote by $Z_{gr}(A)$ its \textit{graded center}, i.e., the largest $G$-graded subalgebra of $Z(A)$.

\begin{thm}[\cite{balaba} Proposition 1]\label{pregradedposner} Let $A$ be a PI $G$-prime algebra, $Z(A)$ the center of $A$ and $S$ the set of homogeneous regular elements of $Z(A)$. Then:
\begin{enumerate}
\item $S=h(Z(A));$
\item the algebra of quotients $A_S$ is a PI $G$-prime algebra;
\item $Z_{gr}(A_S)=Z_{gr}(A)_S.$
\end{enumerate}\end{thm}

\begin{thm}[\cite{balaba} Theorem 5]\label{gradedposner} Let $A$ be a PI $G$-prime algebra and $A_c$ the algebra of central quotients of $A$. Then:
\begin{enumerate}
\item $A_c$ is finite dimensional graded-simple over its graded center $Z$ and $Z$ is the graded field of quotients of $Z_{gr}(A);$
\item $A_c$ is the graded algebra of quotients of $A$;
\item $A$ and $A_c$ satisfy the same identities.
\end{enumerate}\end{thm}

It is easy to see that a $G$-graded domain is a $G$-prime algebra, then by Theorems \ref{pregradedposner} and \ref{gradedposner} we have any PI graded domain can be embedded in a graded prime algebra of central quotients.

It is not difficult to see that $Gen^{\mathbb{Z}_n}(A)$ is a PI ${\mathbb{Z}_n}$-prime algebra, then it admits a ${\mathbb{Z}_n}$-graded algebra of central quotients $Q(A)$. 

We have the following.

\begin{proposition}\label{proposition2}
Let $A=M_n(K)$ be endowed with the $\Z_n$-grading of Vasilovsky, then $Q(A)$ is a $\Z_n$-graded division algebra.
\end{proposition}
\proof
Notice that a homogeneous element in $Q(A)$ has the form $\sum_{g_i^{-1}g_j=g}p_{ij}^{(r)}E_{ij}$, where $p_{ij}$ is a non-zero polynomial in  $K[Y]$, then it is regular, so it invertible in $Q(A)$. 
\endproof

The next will be used later on. See the book \cite{rowenkanelbelov} for more details. 

\begin{definition}
We shall call \textit{expediting $G$-graded algebra}, the algebra generated by the graded generic algebras and the traces of their elements of degree $1$.
\end{definition}

In the light of Proposition \ref{proposition2} we have the next result.

\begin{proposition}\label{proposition1}
The expediting graded algebra endowed with the $\Z_n$-grading of Vasilovsky is a $\Z_n$-graded domain which can be embedded in the $\Z_n$-graded division algebra of central quotients of the graded algebra of generic matrices.
\end{proposition}
\proof
By the graded analogue of Theorem J of \cite{rowenkanelbelov} we have the trace functions of even elements of the graded generic matrix algebra $A$ can be written as the ratio of two polynomials taking values on the graded center of $Q(A)$. Hence any trace of degree 1 elements belongs to $Q(A)$ and we are done.
\endproof

We shall recall some basic topological results that will be useful further in the text. A topological space will be denoted by a pair $(X,\tau_X)$, where $X$ is the underlying space and $\tau_X$ is a topology on $X$, i.e., the set of its open sets. Of course, any closed set of a given topology is the complementary set of an open set and viceversa.
    
    We consider a finite dimensional vector space $V$ over a field $K$. Indeed, $V\cong K^n$ as a vector space, where $n$ equals the dimension of $V$ over $K$. Consider now the algebra of commutative polynomials in $n$ variables $K[X_n]:=K[x_1,\ldots,x_n]$ and choose a set $S\subseteq K[X_n]$. We define \[V(S):=\{\textit{$\overline{a}=(a_1,\ldots,a_n)\in K^n|f(a_1,\ldots,a_n)=0$ for every $f\in S$\}}.\] The set of all $V(S)$, $S\subseteq K[X_n]$ is the set of closed sets of a topology on $V$ called \textit{Zariski's topology}. Recall any finite set of $V$ is a closed set in the Zariski's topology of $V$. Furthermore, the finite sets are the sole closed sets if $V\cong K$.
    
    Given two topological spaces $(X,\tau_X)$, $(Y,\tau_Y)$ and a function $f:X\rightarrow Y$, $f$ is said to be \textit{continuous} if $f^{-1}(U)\in \tau_X$ for every $U\in\tau_Y$ whereas $f$ is said to be \textit{open} if $f(W)\in\tau_Y$ for every $W\in\tau_X$. Furthermore, in a topological space $(X,\tau_X)$ a set $S$ is said \textit{dense} if $\overline{S}=X$, where $\overline{S}$ denotes the \textit{closure} of $S$, i.e., the least closed set containing $S$. Notice that any open set in the Zariski's topology is dense as well as any set containing a dense set in any topological space. We also have the next result.
    
    \begin{lemma}\label{reverse}
        Let $(X,\tau_X)$, $(Y,\tau_Y)$ be topological spaces and $f:X\rightarrow Y$ a continuous and open function. Then for every dense set $S$ of $(Y,\tau_Y)$ we have $f^{-1}(S)$ is dense in $(X,\tau_X)$. 
    \end{lemma}
 
 Now we come back to images of graded polynomials.  Let us set $N$ the set of non-nilpotent matrices of size two and let $M\in N$. Then at least one of the eigenvalues of $M$, $\lambda_1,\lambda_2$, is non-zero. Hence we are allowed to consider the \textit{ratio} of the eigenvalues of $M$ that is $0$ if one between $\lambda_1$ and $\lambda_2$ is 0 whereas it is well defined as $\lambda_1/\lambda_2$ (up to taking reciprocals). We say two non-nilpotent matrices \textit{have different ratios of their eigenvalues} if their ratios of eigenvalues are not equal nor reciprocal. Then we define a map $\Pi:N\rightarrow K$ such that \[\Pi(M)=\left\{\begin{array}{ll}
    0   & \textit{if 0 is an eigenvalue of } M \\
     \lambda_1/\lambda_2+\lambda_2/\lambda_1  & \textit{otherwise}.
 \end{array}\right.\]
 Notice that if $0$ is not an eigenvalue of $M$, then \begin{equation}\label{eigen}\lambda_1/\lambda_2+\lambda_2/\lambda_1=-2+tr(M)^2/det(M).\end{equation}
 
We recall that a polynomial $f=f(x_1,\ldots,x_n)$, is said to be \textit{semi-homogeneous} with nonzero weighted degree $d$ if, letting $d_i$ be the usual degree of $x_i$ in $f$, there exist weights $w_1$, $w_2$, $\dots$, $w_n$ such that in each summand of $f$ we have $w_1d_1 + w_2d_2 +\cdots+ w_nd_n = d$. As pointed out in \cite[Lemma 2]{K-BMR}, if $K$ is closed under $d$-roots, then the image of a semi-homogeneous polynomial is invariant under scalar multiplication.

The next result, the main result of the section, give a description of images of semi-homogeneous $\mathbb{Z}_2$-graded polynomials on $M_2(K)$. Our approach is strongly based to that of \cite{K-BMR}. The difference here is that we need to split the proof in two cases depending on the homogeneous degree of $f$.
 
\begin{thm}
Let $K$ be a field closed under $d$-roots and let $f\in K\langle X \rangle$ be a semi-homogeneous $\Z_2$-graded polynomial of weighted degree $d$. Then, the image of $f$ evaluated on the algebra $M_2(K)$ of $2\times2$ matrices over the  field $K$ is one of the following: \[\{0\},\ \ K,\ \  sl_2(K)_0,\ \ M_2(K)_0,\ \ M_2(K)_1,\ \ D_0,\ \  D_1,\] where $D_i$ is a dense set in the Zariski's topology defined on $M_2(K)_i$.

\end{thm}

\begin{proof}
    Let $f=f(x_1,\ldots,x_m)$ be a semi-homogeneous polynomial of homogeneous degree 0. 
    
    Assume first that every element of the image of $f$ has a fixed ratio $r$ of its eigenvalues. Suppose further $f$ is not an identity and $r\neq\pm 1$. Then the eigenvalues $\lambda_1$, $\lambda_2$ of a certain matrix $f(a_1,\ldots,a_m)$ are linear functions of $tr(f(a_1,\ldots,a_m))$ so they belong to the expediting $\Z_2$-graded algebra that is a graded domain because of Proposition \ref{proposition1}. Notice now $f-\lambda_1I$ and $f-\lambda_2I$ are non-zero homogeneous elements of degree 0 and their product is 0 because of Cayley-Hamilton's Theorem that is an absurd. This forces $r$ being $1$ or $-1$.
    
    Suppose $r=1$, then $f$ is a central polynomial and $Im(f)=K$. Assume $r=-1$, then $Im(f)$ turns out to be $sl_2(K)_0$.

    Now assume there are homogeneous matrices $a_1,\ldots,a_m$, $b_1,\ldots,b_m$ so that $f(a_1,\ldots,a_m)$ and $f(b_1,\ldots,b_m)$ have different ratios of eigenvalues. We consider the map $p:K\longrightarrow M_2(K)_0$ given by \[p(t):=f(ta_1+(1-t)b_1,\ldots,ta_m+(1-t)b_m)\] 
    and set $P=\{p(t)|t\in K\}$. Notice $p$ is a polynomial function whose image $P$ is contained in the image of $f$. Let us denote by $\tilde{h}$ the function $\Pi\circ p:K\rightarrow K$. 
    
    If we endow $K$ and $M_2(K)_0$ with their Zariski's topology, then $\Pi$ and $p$ turn out to be continuous and open. Hence $\tilde{h}$ is continuous and open as well. Recall that $\tilde{h}:K \longrightarrow K$ is a rational function since by equation (\ref{eigen}),
    $\tilde{h} =tr^2(p)/det(p)-2$.
    
    Now write $\tilde h$ as $A(t)/B(t)$, where $A(t)$ and $B(t)$ are polynomial functions of degree less than or equal to $2\deg(f)$ and cannot be written as ratio of polynomials of smaller degrees. Now we are going to study $\tilde{h}(K)$. 
    
    Observe that $c\in K$ belongs to $\tilde{h}(K)$ if and only if there exists $t\in K$ such that $A(t)-cB(t)=0$. Of course, only a finite number of elements of $K$ does not belong to $\tilde{h}(K)$, since the common roots of $A(t)$ and $B(t)$ is a finite set. This means $\tilde{h}(K)$ is open and then dense in the Zariski's topology of $K$. Since $\Pi$ is open and continuous, $\Pi^{-1}(\tilde{h}(K)) = P$ is dense in $M_2(K)_0$ by Lemma \ref{reverse}. Finally, because $P\subseteq Im(f)$, we get $Im(f)$ is dense in $M_2(K)_0$.
    
    We study now the case in which $f=f(x_1,\ldots,x_m)$ is a semi-homogeneous polynomial of homogebeous degree 1. Observe in this case the only possible ratios are $0$ and $-1$. Assume there are homogeneous matrices having different ratios of eigenvalues.  We have the following cases: 
    \begin{enumerate}
        \item $\left(\begin{array}{cc}
           0  &  r\\
            s & 0
        \end{array}\\
        \right)$,\ \ $\left(\begin{array}{cc}
           0  &  a\\
            0 & 0
        \end{array}\\
        \right)$ belong to $Im(f)$, where $r,s,a\neq0$;
        \item $\left(\begin{array}{cc}
           0  &  r\\
            s & 0
        \end{array}\\
        \right)$,\ \ $\left(\begin{array}{cc}
           0  &  0\\
            b & 0
        \end{array}\\
        \right)$ belong to $Im(f)$, where $r,s,b\neq0$.
    \end{enumerate}
In both cases, by Lemma \ref{cone}, we may obtain any element of $M_2(K)_1$ with a suitable conjugation by $E_{12}+E_{21}$ or by some diagonal matrix and scalar multiplication (recall that since $K$ is closed under $d$-roots, the image of $f$ is closed under scalar multiplication). This implies $Im(f)=M_2(K)_1$.  

Assume now every element in the image of $f$ has fixed ratio $r=-1$. Then by Lemma \ref{cone}, and by the fact that $Im(f)$ is closed under scalar multiplication, applying suitable conjugations by diagonal matrices, we obtain that
\[Im(f) = \left\{\left(\begin{array}{cc}
    0 & 0 \\
    0 & 0
\end{array}\right)\right\} \cup \left\{\left(\begin{array}{cc}
   0  & r \\
    s & 0
\end{array}\right), r,s\neq0\right\}.\]
If we set $S=\{xy\}\subset K[x,y]$ one can easily see that $Im(f)=V(S)^c\cup \left\{\left(\begin{array}{cc}
   0  & 0 \\
   0  & 0
\end{array}\right)\right\}$. Now observe $V(S)^c$ being the complementary set of a closed set is open, then it is dense in the Zariski's topology of $M_2(K)_1$. Because $Im(f)$ contains a dense set, it is dense as well.
\end{proof}

\section{Funding}

This work was supported by São Paulo Research Foundation (FAPESP), grant 2018/23690-6.

\bibliographystyle{abbrv}
\bibliography{ref}

\begin{thebibliography}{10}

\bibitem{AM}
A.~A. Albert and B.~Muckenhoupt.
\newblock On matrices of trace zeros.
\newblock {\em Michigan Math. J.}, 4:1--3, 1957.

\bibitem{Azevedo}
S.~S. Azevedo.
\newblock Graded identities for the matrix algebra of order {$n$} over an
  infinite field.
\newblock {\em Comm. Algebra}, 30(12):5849--5860, 2002.

\bibitem{balaba}
I.~N. Balaba.
\newblock Graded prime {PI}-algebras.
\newblock {\em Fundam. Prikl. Mat.}, 9(1):19--26, 2003.

\bibitem{Brandao}
A.~Brand\~{a}o, Jr.
\newblock Graded central polynomials for the algebra {$M_n(K)$}.
\newblock {\em Rend. Circ. Mat. Palermo (2)}, 57(2):265--278, 2008.

\bibitem{div1}
O.~M. Di~Vincenzo.
\newblock On the graded identities of {$M_{1,1}(E)$}.
\newblock {\em Israel J. Math.}, 80(3):323--335, 1992.

\bibitem{Dykema_Klep}
K.~J. Dykema and I.~Klep.
\newblock Instances of the {K}aplansky-{L}vov multilinear conjecture for
  polynomials of degree three.
\newblock {\em Linear Algebra Appl.}, 508:272--288, 2016.

\bibitem{PlamenPedro}
P.~Fagundes and P.~Koshlukov.
\newblock Images of multilinear graded polynomials on upper triangular matrix
  algebras.
\newblock {\em arXiv:2205.10698v1}, 2022.

\bibitem{Fagundes}
P.~S. Fagundes.
\newblock The images of multilinear polynomials on strictly upper triangular
  matrices.
\newblock {\em Linear Algebra Appl.}, 563:287--301, 2019.

\bibitem{GargatedeMello}
I.~G. Gargate and T.~C. de~Mello.
\newblock Images of multilinear polynomials on n × n upper triangular matrices
  over infinite fields.
\newblock {\em Israel J. Math.}, to appear, 2022.

\bibitem{K-BMR}
A.~Kanel-Belov, S.~Malev, and L.~Rowen.
\newblock The images of non-commutative polynomials evaluated on {$2\times2$}
  matrices.
\newblock {\em Proc. Amer. Math. Soc.}, 140(2):465--478, 2012.

\bibitem{K-BMR3}
A.~Kanel-Belov, S.~Malev, and L.~Rowen.
\newblock The images of multilinear polynomials evaluated on {$3\times 3$}
  matrices.
\newblock {\em Proc. Amer. Math. Soc.}, 144(1):7--19, 2016.

\bibitem{survey}
A.~Kanel-Belov, S.~Malev, L.~Rowen, and R.~Yavich.
\newblock Evaluations of noncommutative polynomials on algebras: methods and
  problems, and the {L}'vov-{K}aplansky conjecture.
\newblock {\em SIGMA Symmetry Integrability Geom. Methods Appl.}, 16:Paper No.
  071, 61, 2020.

\bibitem{rowenkanelbelov}
A.~Kanel-Belov and L.~H. Rowen.
\newblock {\em Computational aspects of polynomial identities}, volume~9 of
  {\em Research Notes in Mathematics}.
\newblock A K Peters, Ltd., Wellesley, MA, 2005.

\bibitem{Kemer}
A.~R. Kemer.
\newblock {\em Ideals of identities of associative algebras}, volume~87 of {\em
  Translations of Mathematical Monographs}.
\newblock American Mathematical Society, Providence, RI, 1991.
\newblock Translated from the Russian by C. W. Kohls.

\bibitem{Kra}
I.~Kra and S.~R. Simanca.
\newblock On circulant matrices.
\newblock {\em Notices Amer. Math. Soc.}, 59(3):368--377, 2012.

\bibitem{Kulyamin}
V.~V. Kulyamin.
\newblock Images of graded polynomials in matrix rings over finite group
  algebras.
\newblock {\em Uspekhi Mat. Nauk}, 55(2(332)):141--142, 2000.

\bibitem{Wang_nxn}
Y.~Luo and Y.~Wang.
\newblock On {F}agundes-{M}ello conjecture.
\newblock {\em J. Algebra}, 592:118--152, 2022.

\bibitem{MalevM}
S.~Malev.
\newblock The images of non-commutative polynomials evaluated on {$2\times 2$}
  matrices over an arbitrary field.
\newblock {\em J. Algebra Appl.}, 13(6):1450004, 12, 2014.

\bibitem{MalevQ}
S.~Malev.
\newblock The images of noncommutative polynomials evaluated on the quaternion
  algebra.
\newblock {\em J. Algebra Appl.}, 20(5):Paper No. 2150074, 8, 2021.

\bibitem{MalevJ}
S.~Malev, R.~Yavich, and R.~Shayer.
\newblock Evaluations of multilinear polynomials on low rank {J}ordan algebras.
\newblock {\em Comm. Algebra}, 50(7):2840--2845, 2022.

\bibitem{SantuloYukihide}
E.~Santulo and F.~Yukihide~Yasumura.
\newblock On the image of polynomials evaluated on incidence algebras: a
  counter-example and a solution.
\newblock {\em arXiv:1902.08116v1}, 2019.

\bibitem{Shoda}
K.~Shoda.
\newblock Einige {S}\"{a}tze \"{u}ber {M}atrizen.
\newblock {\em Jpn. J. Math.}, 13(3):361--365, 1937.

\bibitem{Vasilovsky}
S.~Y. Vasilovsky.
\newblock {$Z_n$}-graded polynomial identities of the full matrix algebra of
  order {$n$}.
\newblock {\em Proc. Amer. Math. Soc.}, 127(12):3517--3524, 1999.

\end{thebibliography}

\end{document}